\documentclass[twoside,11pt]{amsart}

\usepackage{amsmath}
 \usepackage{epsfig} 

\newtheorem{theorem}{Theorem}[section]
\newtheorem{corollary}{Corollary}[section]

\newtheorem{lemma}{Lemma}[section]
\newtheorem{proposition}{Proposition}[section]

\newtheorem{remark}{Remark}[section]

\numberwithin{equation}{section}

\def\vv{{\mathbf v}}

\def\vn{\widetilde{\vv}}
\def\vp{\widehat{\vv}}
\def\vu{{\mathbf u}}

\def\be{\begin{equation}}
\def\ee{\end{equation}}
\def\ba{\begin{array}}
\def\ea{\end{array}}
\def\ul{{{\mathbf u}_1}}   
\def\ho{{h_1}}
\def\ull{{{\mathbf u}_2}}
\def\v{\ull}

\def\hll{{h_2}}  
\def\rholl{{\rho_2}}

\def\pl{{p_2}}
\def\Vl{{\bf U_2}}
\def\U{{\bf U}}
\def\W{{\bf E}}
\def\F{{\bf F}}
\def\G{{\bf G}}
\def\bR{{\bf R}}
\def\R{{\Bbb R}}
\def\T{{\Bbb T}}

\def\C{C_0}
\def\Cinf{\widehat{C}_0}
\def\cn{\cdot\nabla}
\def\nc{\nabla\cdot}

\def\p{\partial}
\def\pt{\p_t}
\def\Rossby{\tau}
\def\Rossbyi{\frac{1}{\Rossby}}
\def\Froude{\sigma}
\def\Froudei{\frac{1}{\Froude}}
\def\etj{e^{tJ/\Rossby}}

\def\dde{\delta}
\def\At{{e^{\C t}\dde\over1-e^{\C t}\dde }}
\def\bea{\begin{equation}\left\{\begin{array}}
\def\eea{\end{array}\right.\end{equation}}
\def\lesssim{\stackrel{{}_<}{{}_\sim}}
\def\lsm{\lesssim_m}

\def\ppi{\widetilde{p}}
\def\Cg{C_{\gamma}}
\def\ra{\rightarrow}

\begin{document}

\title[Long time existence of the rapidly rotating Euler equations]
{Long time existence of smooth solutions for\\
the rapidly rotating shallow-water and  Euler equations}

\author[Bin Cheng]{Bin Cheng}
\address[Bin Cheng]{\newline
    Department of Mathematics\newline
    Center of Scientific Computation And Mathematical Modeling (CSCAMM)\newline
        University of Maryland\newline
        College Park, MD 20742 USA}
\email []{bincheng@cscamm.umd.edu}
\urladdr{http://www.math.umd.edu/\~{}bincheng}
\author[Eitan Tadmor]{Eitan Tadmor}
\address[Eitan Tadmor]{\newline
        Department of Mathematics,  Institute for Physical Science and Technology\newline
        and Center of Scientific Computation And Mathematical Modeling (CSCAMM)\newline
        University of Maryland\newline
        College Park, MD 20742 USA}
 \email[]{tadmor@cscamm.umd.edu}
\urladdr{http://www.cscamm.umd.edu/\~{}tadmor}


\date\today


\thanks{\textbf{Acknowledgment.} Research was supported in part by NSF grant 04-07704 and ONR grant N00014-91-J-1076.}


\begin{abstract}
We study the stabilizing effect of rotational  forcing in the nonlinear setting of two-dimensional shallow-water and more general models of compressible Euler equations. In \cite{LT04} we have shown that the pressureless version of these equations admit global smooth solution for a large set of sub-critical initial configurations. In the present work we prove that when rotational force dominates the pressure, it \emph{prolongs} the life-span of smooth solutions for 
$t\lesssim \ln(\dde^{-1})$; here $\dde \ll 1$ is the  ratio of the pressure gradient  measured by the inverse squared Froude number, relative to the dominant rotational forces  measured by  the inverse Rossby number. Our study reveals a ``nearby'' periodic-in-time approximate solution in the small $\dde$ regime, upon which hinges the long time existence of the exact smooth solution. These results are in agreement with the close-to periodic dynamics observed in the ``near inertial oscillation'' (NIO) regime which follows oceanic storms. Indeed, our results indicate the existence of smooth, ``approximate periodic'' solution for a time period of \emph{days}, which is the relevant time period  found in NIO obesrvations.
\end{abstract}
\maketitle
\tableofcontents

\section{Introduction and statement of main results}\label{Intro}
\setcounter{figure}{0}
\setcounter{equation}{0}

We are concerned here with two-dimensional systems of nonlinear Eulerian equations driven by pressure and rotational forces. It is well-known that in the absence of rotation, these equations experience a finite-time breakdown: for generic smooth initial conditions, the corresponding solutions will lose  $C^1$-smoothness due to shock formation. 
The presence of rotational forces, however, has a stabilizing effect.
In particular, the pressureless version of these equations admit  global smooth solutions for a large set of so-called sub-critical initial configurations, \cite{LT04}.
It is therefore a natural extension to investigate the balance between the regularizing effects of rotation vs. the tendency of pressure to enforce finite-time breakdown
(we mention in passing the recent work \cite{TW06} on a similar regularizing balance of different competing forces in the  1D Euler-Poisson equations). In this paper we prove the long-time existence of rapidly rotating flows  characterized by ``near-by'' periodic flows.
Thus, rotation \emph{prolongs} the life-span of smooth solutions over increasingly long time periods, which grow longer as the rotation forces become more dominant over pressure.
    
Our model problem is the Rotational Shallow Water (RSW) equations. This system of equations   models large scale geophysical motions in a thin layer of fluid under the influence of the Coriolis  rotational forcing, (e.g. \cite[\S 3.3]{Pe92}, \cite[\S 2.1]{GC87}),
\begin{subequations}\label{sub:phys}
\begin{eqnarray} 
\label{phys2} \pt h+\nabla\cdot(h\vu)&=&0,\\
\label{phys1}\pt \vu+\vu\cn \vu+g\nabla h-f\vu^\perp&=&0.
\end{eqnarray}
\end{subequations}
It governs the unknown  velocity field $\vu:=\big(u^{(1)}(t,x,y),u^{(2)}(t,x,y)\big)$ an 
height $h:=h(t,x,y)$, where $g$ and $f$ stand for the gravitational constant and  the Coriolis frequency. Recall that  equation (\ref{phys2}) observes the conservation of mass and equations (\ref{phys1}) describe balance of momentum by the pressure gradient,  $g\nabla h$, and  rotational forcing, $f\vu^\perp:=f\big(u^{(2)},-u^{(1)}\big)$. 

For convenience, we  rewrite the system (\ref{sub:phys}) in terms of  rescaled, nondimensional variables. To this end, we introduce the characteristic scales, $H$ for total height $h$, $D$ for height fluctuation $h-H$, $U$ for velocity $u$,  $L$ for spatial length and correspondingly $L/U$ for time, and we make the change of variables
\[
\vu=\vu'\left(\frac{t'L}{U},x'L,y'L\right)U, \qquad h=H+h'\left(\frac{t'L}{U},x'L,y'L\right)D.
\]
Discarding all the primes, we arrive at a nondimensional system,
\begin{subequations}
\begin{eqnarray*}
\label{adim2}\pt h+\vu\cn h+\left({H\over D}+h\right)\nabla\cdot \vu&=&0,\\
\label{adim1}\pt \vu+\vu\cn \vu+{gD\over U^2}\nabla h-{fL\over U}\vu^\perp&=&0.
\end{eqnarray*}
\end{subequations}

We are  concerned here with the regime where the pressure gradient and compressibility are of the same order, ${\displaystyle {gD\over U^2}\approx{H\over D}}$. Thus we arrive at the (symmetrizable) RSW system,
\begin{subequations}\label{RSW} 
\begin{eqnarray}  
\pt h+\vu\cdot\nabla h+\left(\Froudei+h\right)\nabla\cdot \vu&=&0, \label{RSW22}\\
\pt \vu+\vu\cdot\nabla \vu+\Froudei\nabla h-\Rossbyi J\vu&=&0,\label{RSW11}.
\end{eqnarray}  
Here $\Froude$ and $\Rossby$, given by
\begin{equation}\label{RSW33}
\Froude:=\frac{U}{\sqrt{gH}}, \quad   \Rossby:=\frac{U}{fL},
\end{equation}
\end{subequations}
are respectively, the Froude  number measuring the inverse pressure forcing and 
the Rossby number measuring the inverse rotational forcing. We use $J$ to denote the $2\times2$ rotation matrix ${\displaystyle J:=\left(\begin{array}{lr} 0 & 1\\-1 & 0\end{array}\right)}$.

To trace their long-time behavior, we approximate (\ref{RSW22}), (\ref{RSW11})  with the successive iterations,
\begin{subequations}\label{suc}
\begin{eqnarray}
\label{hj}
\pt h_j + \vu_{j-1}\cn h_j +\left( \Froudei+h_j\right)\nabla\cdot \vu_{j-1} &= &0, \qquad   j=2,3, \ldots\\
\label{uj}
\pt \vu_j+\vu_j\cn \vu_j+\Froudei \nabla h_{j}-\Rossbyi J\vu_j & = & 0, \qquad j=1,2, \ldots,
\end{eqnarray}
subject to initial conditions, $h_j(0,\cdot)=h_0(\cdot)$ and $\vu_j(0,\cdot)=\vu_0(\cdot)$. Observe that, given $j$, (\ref{suc}) are only weakly coupled through the dependence of $\vu_j$ on $h_j$, so that we only need to specify the initial height $\ho$. Moreover, for $\Froude\gg \Rossby$, the momentum equations (\ref{RSW11}) are ``approximately decoupled'' from the mass equation (\ref{RSW22}) since rotational forcing is substantially dominant over pressure forcing.
Therefore, a first approximation of constant height function will enforce this decoupling, serving as the starting point of the above iterative scheme,

\begin{equation}
\ho\equiv\mbox{constant}.
\end{equation}
\end{subequations}
This, in turn, leads to the  first approximate velocity field, $\ul$, satisfying the pressureless equations,
\begin{equation}\label{ipre}
\pt \ul+\ul\cn \ul-\Rossbyi J\ul=0, \quad\ul(0,\cdot)=\vu_0(\cdot).
\end{equation} 
Liu and Tadmor \cite{LT04} have shown  that there is a ``large set'' of so-called sub-critical initial configurations $\vu_0$, for which the pressureless equations (\ref{ipre}) admit global smooth solutions. Moreover, the pressureless velocity  $\ul(t,\cdot)$ is in fact $2\pi\Rossby$-periodic in time. The regularity of $\ul$ is discussed  in Section \ref{sec:2}.

Having the pressureless solution, $(\ho\equiv\mbox{constant},\ul)$ as a first approximation for the RSW solution $(h,\vu)$, 
in Section \ref{sec:3} we introduce an improved approximation  of the RSW equations, $(\hll,\ull)$, which solves an ``adapted'' version of the second iteration ($j=2$) of (\ref{suc}). This improved approximation  satisfies a specific linearization of the RSW equations around the pressureless velocity $\ul$, with only a one-way coupling between the momentum and the mass equations. Building on the regularity and periodicity of the pressureless velocity $\ul$,  we show that  the solution of this linearized system subject to sub-critical initial data  $(h_0,\vu_0)$, is globally smooth; in fact, both $\hll(t,\cdot)$ and $\ull(t,\cdot)$ retain $2\pi\Rossby$-periodicity in time.

Next, we turn to estimate the deviation between the solution of the linearized RSW system, $(\hll,\ull)$, and the solution of the full RSW system, $(h,\vu)$.  To this end,  we introduce a new non-dimensional parameter 
\[
\dde:={{\Rossby}\over\Froude^2} = {\frac{g H}{f L U}},
\]
measuring the {\it relative} strength of rotation vs. the pressure forcing, and we assume that rotation is the dominant forcing in the sense that $\dde \ll 1$. 
Using the standard energy method we show  in Theorem \ref{thm:main} and its corollary that, starting with $H^m$ sub-critical initial data, the RSW solution  $\big(h(t,\cdot),\vu(t,\cdot)\big)$ remains sufficiently close to $\big(\hll(t,\cdot), \ull(t,\cdot)\big)$ in the sense that,
 
\[
\|h(t,\cdot)-\hll(t,\cdot)\|_{H^{m-3}} + \|\vu(t,\cdot)-\vu_2(t,\cdot)\|_{H^{m-3}} \lesssim
{e^{C_0t}\dde\over(1-{e^{C_0t}\dde)^2}},
\]where constant $C_0=\widehat{C}_0(m,|\nabla\vu_0|_\infty,|h_0|_\infty)\cdot\|\vu_0,h_0\|_m$.
In particular,  we conclude that for a large set of sub-critical initial data, the RSW equations (\ref{RSW}) admit  smooth, ``approximate periodic'' solutions for long time, $t\leq t_{\dde}:=\ln(\dde^{-1})$, in the rotationally dominant regime $\dde\ll1$. 

We comment that our formal notion of ``approximate periodicity'' emphasizes the existence of a periodic approximation $(\hll,\ull)$ nearby the actual flow $(h,\vu)$, with an up-to $O(\dde)\ll1$-error for sufficiently long time. Therefore, strong rotation stabilizes the flow by imposing on it approximate periodicity, which in turn postpones finite time breakdown of classical solutions for a long time. A convincing example is provided by the so called ``near-inertial oscillation'' (NIO) regime, which is observed during the days that follow oceanic storms, e.g. \cite{YB97}. These NIOs are triggered when storms pass by (large $U$'s) and only a thin layer of the oceans is reactive (small aspect ratio $H/L$), corresponding to $\dde={gH\over fLU}\ll1$.  Specifically, with Rossby number $\Rossby\sim {\mathcal O}(0.1)$ and  Froude number $\Froude\sim {\mathcal O}(1)$ we find $\dde \sim 0.1$, which yield the existence of smooth, ``approximate periodic'' solution for $t \sim 2$ days. We note that  the clockwise rotation of cyclonic storms on the Northern Hemisphere produce  negative vorticity which is a preferred scenario of the sub-critical condition (\ref{CT}).
Our results are consistent with the observations regarding the stability and approximate periodicity of the NIO regime.


Next, we generalize our result to Euler systems describing the isentropic gasdynamics, in Section \ref{sec:5}, and ideal gasdynamics, in Section \ref{sec:6}. We regard these two systems as successive generalizations of the RSW system using the following formalism,
\begin{subequations}\label{gEuler}
\begin{eqnarray}
\pt \rho+\nc(\rho \vu)&=&0, \label{gEuler1}\\
\pt \vu+\vu\cn \vu+\rho^{-1}\nabla \ppi(\rho,S)&=&f J\vu,
 \label{gEuler2}\\
\pt S+\vu\cn S&=&0.\label{gEuler3}
\end{eqnarray}
Here, the physical variables $\rho$, $S$ are respectively the density and entropy. We use $\ppi(\rho,S)$ for the gas-specific pressure law relating pressure to density and entropy. For the ideal gasdynamics, the pressure law is given as $\ppi:=A\rho^\gamma e^S$ where $A,\gamma$ are two gas-specific physical constants. The isentropic gas equations correspond to constant $S$, for which the entropy equation (\ref{gEuler3}) becomes redundant. Setting $A=g,\gamma=2$ yields the RSW equations with $\rho$ playing the same role as $h$.

The general Euler system (\ref{gEuler}) can be symmetrized by introducing a ``normalized'' pressure function,
\[
p:={\sqrt{\gamma}\over\gamma-1}\ppi^{\gamma-1\over2\gamma}(\rho,S),
\]
and by replacing the density equation (\ref{gEuler1}) with a pressure equation,
\begin{equation}\label{gEuler4}
\pt p+\vu\cn p+{\gamma-1\over2}p\nc\vu=0.
\end{equation}
\end{subequations}
  We then nondimensionalize the above system (\ref{gEuler2}), (\ref{gEuler3}) and (\ref{gEuler4})  into 
\begin{eqnarray*}
\pt p+\vu\cn p+{\gamma-1\over2}\left(\Froudei+p\right)\nc\vu&=&0,\\
\pt\vu+\vu\cn\vu+{\gamma-1\over2}\left(\Froudei+p\right)e^{\Froude S}\nabla p&=&\Rossbyi J\vu,\\ 
\pt S+\vu\cn S&=&0.
\end{eqnarray*}
The same methodology introduced for the RSW equations still applies to the more general Euler system, independent of the pressure law. In particular, our first approximation, the pressureless system, remains the same as in (\ref{ipre}) since it ignores any effect of pressure. We then obtain the second approximation $(\pl,\ull,S_2)$ (or $(\pl,\ull)$ in the isentropic case) from a specific linearization around the pressureless velocity $\ul$. Thanks to the fact that $h$, $p$ and $S$ share a similar role as passive scalars transported by $\vu$, the same regularity and periodicity argument can be employed for $(\pl,\ull,S_2)$ in these general cases as for $(h_2,\ull)$ in the RSW case. The energy estimate, however, needs careful modification for the ideal gas equations due to additional nonlinearity. Finally, we conclude in Theorem 4.2 and 4.3 that, in the rotationally dominant regime $\dde\ll1$, the exact solution stays ``close'' to the globally smooth, $2\pi\Rossby$-periodic approximate solution $(\pl,\ull, S_2)$ for long time in the sense that, starting with $H^m$ sub-critical data, the following estimate holds true for time $t\lesssim \ln(\dde^{-1})$,
\[
\|p(t,\cdot)-\pl(t,\cdot)\|_{m-3}+
\|\vu(t,\cdot)-\ull(t,\cdot)\|_{m-3}+
\|S(t,\cdot)-S_2(t,\cdot)\|_{m-3}<\At.
\]

Our results confirm the stabilization effect of rotation in the nonlinear setting, when it interacts with the slow components of the system, which otherwise tend to destabilize of the dynamics. The study  of such interaction is essential to the  understanding of rotating dynamics, primarily to geophysical flows.   
We can mention only few works from the vast literature available on this topic, and we refer the reader to the recent book of Chemin et. al., \cite{CDGG06} and the references therein, for a state-of-the art of the mathematical theory for rapidly rotating  flows. Embid and Majda \cite{EM96, EM98} studied the singular limit of RSW equations under the two regimes $\Rossby^{-1}\sim\Froude^{-1}\rightarrow\infty$ and 
$\Rossby^{-1}\sim {\mathcal O}(1),\,\Froude^{-1}\rightarrow\infty$. Extensions to more general skew-symmetric perturbations can be found  in the work of Gallagher, e.g. \cite{Ga98}. The series of works of Babin, Mahalov and Nicolaenko, consult \cite{BMN96, BMNZ97, BMN98, BMN00, BMN02} and references therein,  establish long term stability effects  of the rapidly rotating 3D Euler, Navier-Stokes and primitive equations. Finally, we mention the work of
Zeitlin, Reznik and Ben Jelloul in \cite{ZRB01, ZRB03} which categorizes several relevant scaling regimes and correspondingly, derives formal asymptotics in the nonlinear setting.  

We comment here that the approach pursued in the above literature  relies on identifying the limiting system as $\Rossby\rightarrow0$,  which filters out fast scales. The full system is then approximated to a first order, by this slowly evolving limiting system. A rigorous mathematical foundation along these lines was developed by Schochet \cite{Sc94}, which can be traced back to the earlier works of Klainerman and Majda \cite{KM82} and Kreiss \cite{Kr80} (see also \cite{Ta82}).
The key point was the separation of (linear) fast oscillations from the slow scales. 
The novelty of our approach, inspired by the critical threshold phenomena \cite{LT02}, is to adopt the rapidly oscillating and fully {\it nonlinear} pressureless system as a first  approximation and then consider the full system as a perturbation of this fast scale. This enables us to preserve both slow and fast dynamics, and especially, the rotation-induced time periodicity.

\section{First approximation-- the pressureless system}\label{sec:2}
We consider the pressureless system
\begin{subequations}\label{pre}
\begin{equation}\label{u00}
\pt \ul+\ul\cn \ul-\Rossbyi J\ul =  0,
\end{equation}
subject to initial condition $\ul(0,\cdot)=\vu_0(\cdot)$. 
We begin  by recalling the main theorem in \cite{LT04} regarding the global regularity of the pressureless equations (\ref{u00}).

\begin{theorem}\label{CTthm}
Consider the pressureless equations (\ref{u00}) subject to $C^1$-initial data $\ul(0,\cdot)=u_0(\cdot)$. Then, the solution $\ul(t,\cdot)$ stays $C^1$ for all time if and only if the initial data satisfy the critical threshold condition,
\begin{equation}\label{CT}
\Rossby\omega_0(x)+{\Rossby^2\over2}\eta^2_0(x)<1, \quad \ {\rm for \ all} \ x\in \R^2.
\end{equation}
Here, $\omega_0(x)=-\nabla\times \vu_0(x)=\partial_yu_0-\partial_xv_0$ is the initial vorticity and  $\eta_0(x):=\lambda_1-\lambda_2$ is the (possibly complex-valued) spectral gap associated with the eigenvalues of gradient matrix $\nabla \vu_0(x)$. Moreover, these globally smooth solutions, 
$\ul(t,\cdot)$, are $2\pi\Rossby$-periodic in time.
\end{theorem}
\end{subequations}

In \cite{LT04}, Liu and Tadmor gave two different proofs of (\ref{CT}). One was based on the spectral dynamics of $\lambda_j(\nabla\vu)$; another, was based on the flow map associated with  (\ref{u00}), and here we note yet another version of the latter, based on the Riccati-type equation satisfied by the gradient matrix $M=:\nabla \ul$,
\[
M'+M^2=\Rossby^{-1}JM. 
\]
Here $\{\cdot\}':=\pt+\ul\cn$ denotes differentiation along the particle trajectories
\begin{equation}\label{eq:gamma}
\Gamma_{0}:=\{(x,t) \ | \ \dot{x}(t)=\ul(x(t),t), \ x(t_0)=x_0\}.
\end{equation}
Starting with $M_0=M(t_0,x_0)$, the solution of this  equation along the corresponding trajectory $\Gamma_0$ is given by  
\[
M=\etj \left(I+\Rossby^{-1}J\left(I-\etj\right)M_0\right)^{-1}M_0,
\]
and a straightforward calculation based on the Cayley-Hamilton Theorem (for computing the inverse of a matrix) shows that
\begin{equation}\label{M}
\max_{t,x}|\nabla\ul|=\max_{t,x}|M|=\max_{t,x}\left|{\mbox{polynomial}(\Rossby,\etj,\nabla\vu_0)\over(1-\Rossby\omega_0-{\Rossby^2\over2}\eta^2_0)_+}\right|.
\end{equation}
Thus the critical threshold (\ref{CT}) follows.  The periodicity of $\ul$ is proved upon integrating $\ul'={1\over\Rossby}Ju$ and $x'=\ul$ along particle trajectories $\Gamma_0$. It turns out both $x(t)$ and $\ul(t,x(t))$ are $2\pi\Rossby$ periodic, which clearly implies that $\ul(t,\cdot)$ shares the same periodicity.
It follows that  there exists a critical Rossby number, $\Rossby_c:=\Rossby_c(\nabla \vu_0)$  such that the pressureless solution, $\ul(t,\cdot)$, remains smooth for global time whenever $\Rossby\in(0,\Rossby_c)$. This  emphasizes the stabilization effect of the rotational forcing for a ``large" class of sub-critical initial configurations, \cite[\S1.2]{LT04}.
Observe that the critical threshold, $\Rossby_c$ \emph{need not} be small, and in fact, $\Rossby_c=\infty$ for rotational initial data  such that $\eta_0^2 <0$, $\omega_0<\sqrt{-\eta_0^2}$. We shall always limit ourselves, however,  to a \emph{finite} value of the critical threshold, $\Rossby_c$. 

In the next corollary we show that in fact, the pressureless solution retains higher-order smoothness of the sub-critical initial data.  To this end, we introduce the following notations.

\medskip\noindent
\emph{Notations}.  Here and below, $\|\cdot\|_{m}$ denotes the usual $H^m$-Sobolev norm over the 2D torus ${\T^2}$ and $|\cdot|_\infty$ denotes the $L^\infty$ norm. We abbreviate $a\lsm b$ for $a\leq cb$ whenever the constant $c$ only depends on the dimension $m$. We let $\Cinf$ denote $m$-dependent constants that have possible nonlinear dependence on the initial data $|h_0|_\infty$ and $|\nabla\vu_0|_\infty$. The constant,  $\C:=\Cinf\cdot\|(h_0,\vu_0)\|_m$, will be used for estimates involving Sobolev regularity, emphasizing that $\C$ depends \emph{linearly} on the 
$H^m$-size of initial data, $h_0$ and $\vu_0)\|$,  and possibly nonlinearly on their $L^\infty$-size.

\begin{corollary}\label{CTcor}
Fix an integer $m>2$ and consider the pressureless system (\ref{u00}) subject to sub-critical initial data, $u_0\in H^m$. Then, there exists a critical value $\Rossby_c:=\Rossby_c(\nabla \vu_0) < \infty$  such that for $\Rossby\in(0, \Rossby_c]$ we have, uniformly in time,
\begin{subequations}\label{uo:est} 
\begin{eqnarray}
\label{uo:inf}|\nabla\ul(t,\cdot)|_\infty&\le&\Cinf,\\
\label{uo:m}\|\ul(t,\cdot)\|_{m}&\le&\C. 
\end{eqnarray}
\end{subequations} 
\end{corollary}
\begin{proof}
We recall the expression for $|\nabla\ul|_\infty$ in (\ref{M}). By continuity argument, there exists a value $\Rossby_c$ such that $1-\Rossby\omega_0-{\Rossby^2\over2}\eta^2_0>1/2$ for all $\Rossby\in(0,\Rossby_c)$ , which in turn implies (\ref{uo:inf}) with a constant $\Cinf$ that depends on $|\nabla\vu_0|_\infty$ and $\Rossby_c$ which also replies on the pointwise value of $\nabla\vu_0$.  

Having control on the $L^\infty$ norm of $\nabla\ul$, we employ the standard energy method to obtain the inequality, 
\[
{d\over dt}\|\ul(t,\cdot)\|_{m}\lsm  |\nabla\ul(t,\cdot)|_{L^\infty}\|\ul(t,\cdot)\|_{m}.
\]
Since $\ul(t,\cdot)$ is $2\pi\Rossby$-periodic, it suffices to consider its energy growth over $0\le t<2\pi\Rossby<2\pi\Rossby_c$.  Combining with estimate (\ref{uo:inf}) and solving the above Gronwall inequality, we prove the $H^m$ estimate (\ref{uo:m}).
\end{proof}

\section{Second approximation -- the linearized system}\label{sec:3}
Once we established the global properties of the pressureless velocity $\ul$, it can be used as the starting point for second iteration of (\ref{suc}). We begin with the
approximate height, $\hll$, governed by (\ref{hj}),
\begin{equation}\label{hl:eq}
\pt \hll + \ul\cn\hll + \left(\Froudei+\hll\right)\nabla\cdot\ul=0, \qquad \hll(0,\cdot)=h_0(\cdot).
\end{equation}
Recall that $\ul$ is the solution of the pressureless system (\ref{u00})
subject to \emph{sub-critical} initial data $\vu_0$, so that $\ul(t,\cdot)$ is smooth, $2\pi\Rossby$-periodic in time. 
The following key lemma shows that the periodicity of $\ul$ imposes the same periodicity on passive scalars transported by such $\ul$'s.

\begin{lemma}\label{period}
Let scalar function $w$ be governed by 
\begin{equation}\label{wp}
\pt w+\nabla\cdot(\ul w)=0
\end{equation} 
where $\ul(t,\cdot)$ is a globally smooth, $2\pi\Rossby$-periodic solution of the pressureless equations (\ref{u00}). Then $w(t,\cdot)$ is also $2\pi\Rossby$-periodic.
\end{lemma}

\begin{proof}
Let $\phi:=\nabla\times \ul+\Rossby^{-1}$ denote the so-called relative vorticity. By (\ref{u00}) it satisfies the same equation $w$ does, namely, 
\[
\pt \phi+\nabla\cdot(\ul\phi)=0.
\]
Coupled with (\ref{wp}), it is easy to verify that the ratio $w/\phi$ satisfies a transport equation 
\[
\big(\pt +\ul\cn \big){w\over\phi}=0
\]
which in turn implies that $w/\phi$ remains constant along the trajectories 
$\Gamma_0$ in (\ref{eq:gamma}). But (\ref{u00}) tells us that $\ul'=\frac{J}{\Rossby}\ul$, yielding  $\ul(t,x(t))=e^{{t\over\Rossby}J}\vu_0(x_0)$. We integrate to find,  $x(2\pi\Rossby)=x(0)$, namely, the trajectories come back to their initial positions at $t=2\pi\Rossby$. Therefore 
\[
{w\over\phi}(2\pi\Rossby,x_0)={w\over\phi}(0,x_0) \quad \text {for all} \ x_0\text{{}'s}.
\] 
Since the above argument is time invariant, it implies that $w/\phi(t,\cdot)$ is  $2\pi\Rossby$-periodic. The conclusion follows from the fact that $\ul(t,\cdot)$ and thus $\phi(t,\cdot)$ are $2\pi\Rossby$-periodic.
\end{proof}

Equipped with this lemma we conclude  the following.

\begin{theorem}\label{hl:th}
Consider the  mass equation (\ref{hl:eq}) on a 2D torus, $\T^2$, 
linearized around the pressureless velocity field $\ul$ and
subject to sub-critical initial data $(h_0,\vu_0)\in H^m({\T^2})$ with $m>5$.
It admits a globally smooth solutions, $\hll(t,\cdot)\in H^{m-1}({\T^2})$ 
which is $2\pi\Rossby$-periodic in time, 
and the following upper bounds hold uniform in time,
\begin{subequations}\label{hl:est}
\begin{eqnarray}
\label{hl:inf}|\hll(t,\cdot)|_\infty\le\Cinf\left(1+\frac{\Rossby}{\Froude}\right),\\
\label{hl:m}\|\hll(t,\cdot)\|_{m-1}\le \C \left(1+\frac{\Rossby}{\Froude}\right). 
\end{eqnarray}
\end{subequations}
 \end{theorem}

\begin{proof}
Apply lemma \ref{period} with $w:=\Froude^{-1}+\hll$ to (\ref{hl:eq})  to conclude that $\hll$ is also $2\pi\Rossby$-periodic. We turn to the examine the regularity of $\hll$. First, its $L^\infty$ bound (\ref{hl:inf}) is studied using the $L^\infty$ estimate for scalar transport equations which yields an inequality for $|\hll|_\infty=|\hll(t,\cdot)|_\infty$,
\[
{d\over dt}|\hll|_\infty\le|\nabla\cdot\ul|_\infty(\Froude^{-1}+|\hll|_\infty).
\]
Combined with the $L^{\infty}$ estimate of $\nabla\ul$ in (\ref{uo:inf}), this Gronwall inequality implies
\[
|\hll|_\infty\le e^{\Cinf t}|h_0|_\infty+{1\over\Froude}\left(e^{\Cinf t}-1\right).
\]
As before, due to the $2\pi\Rossby$-periodicity of $\hll$ and the subcritical condition $\Rossby\le\Rossby_c$, we can replace the first $t$ on the right with $\Rossby_c$, the second $t$ with $2\pi\Rossby$, and (\ref{hl:inf}) follows.

For the $H^{m-1}$ estimate (\ref{hl:m}), we use the energy method and the Gagliardo-Nirenberg inequality to obtain a similar inequality for $|\hll|_{m-1}=|\hll(t,\cdot)|_{m-1}$,
\[
{d\over dt}\|\hll\|_{m-1}\lsm|\nabla\ul|_\infty\|\hll\|_{m-1}+\left({1\over\Froude}+|\hll|_\infty\right)\|\ul\|_m.
\]
Applying the estimate on $\ul$ in (\ref{uo:est}) and the $L^\infty$ estimate on $\hll$ in (\ref{hl:inf}), we find the above inequality shares a similar form as the previous one. Thus the estimate (\ref{hl:m}) follows by the same periodicity and sub-criticality argument as for (\ref{hl:inf}). We note by passing the {\it linear} dependence of $\C$ on $\|(h_0,\vu_0)\|_m$.\end{proof}

To continue with the second approximation, we turn to the approximate momentum equation (\ref{uj}) with $j=2$, 
\begin{equation}\label{iterv}
\pt\v+\v\cn\v+{1\over\Froude}\nabla h_2-{1\over\Rossby}\v=0.
\end{equation}

The following \emph{splitting} approach will lead to a simplified linearization of (\ref{iterv}) which is ``close'' to (\ref{iterv}) and still maintains the nature of our methodology. The idea is to treat the nonlinear term and the pressure term in (\ref{iterv}) separately, resulting in two systems for $\vn\approx\v$ and $\vp\approx\v$,
\begin{subequations}\begin{eqnarray}\label{itervn}\pt\vn+\vn\nc\vn-{1\over\Rossby}J\vn&=&0,\\
\label{itervp}\pt\vp+{1\over\Froude}\nabla h_2-{1\over\Rossby}J\vp&=&0,\end{eqnarray}\end{subequations} 
subject to the same initial data $\vn(0,\cdot)=\vp(0,\cdot)=\vu_0(\cdot).$

The first system (\ref{itervn}), ignoring the pressure term, is identified as the pressureless system (\ref{pre}) and therefore is solved as \begin{subequations}\[\vn=\ul,\]
while the second system (\ref{itervp}), ignoring the nonlinear advection term, is solved using the Duhamel's principle,
\[\ba{rcl}\vp(t,\cdot)&=&\etj\left(\vu_0(t,\cdot)-\displaystyle\int_0^t{e^{-sJ/\Froude}\over\Froude}\nabla h_2(s,\cdot)\,ds\right)\\
\displaystyle&\approx&\etj\left(\vu_0(t,\cdot)-\displaystyle\int_0^t{e^{-sJ/\Froude}\over\Froude}\nabla h_2(t,\cdot)\,ds\right)\\
\displaystyle&=&\etj\vu_0(t,\cdot)+\displaystyle{\Rossby\over\Froude}J(I-\etj)\nabla h_2(t,\cdot).\ea\]
Here, we make an approximation by replacing $h_2(s,\cdot)$ with $h_2(t,\cdot)$ in the integrand, which introduces an error of order $\Rossby$, taking into account the $2\pi\Rossby$ period of $h(t,\cdot)$.\end{subequations}

Now, synthesizing the two solutions listed above, we make a correction to $\vp$  by replacing $\etj\vu_0$ with $\ul$. This gives the very form of our approximate velocity field $\ull$ (with tolerable abuse of notations)   
\begin{subequations}\label{sub:ul}
\begin{equation}\label{solu2}
\ull:=\ul+{\Rossby\over\Froude}J(I-\etj)\nabla h_2(t,\cdot).
\end{equation}
A straightforward computation shows that this velocity field, $\ull$, satisfies the following approximate momentum equation,
\begin{equation}\label{iteru2}
\pt\ull+\ul\cn\ull+{1\over\Froude}\nabla h_2-{1\over\Rossby}\ull^\perp=R\end{equation}
where
\begin{equation}\label{Ru}
\ba{rrrl}&R&:=&\displaystyle{\Rossby\over\Froude}J(I-\etj)(\pt+\ul\cn)\nabla h_2(t,\cdot)\\
\mbox{(by (\ref{hl:eq}))}&&=&-\displaystyle{\Rossby\over\Froude}J(I-\etj)\left[(\nabla\ul)^\top\nabla\hll+\nabla((\Froudei+\hll)\nc\ul)\right].\ea
\end{equation}
\end{subequations}

Combining Theorem \ref{hl:th} on $\hll(t,\cdot)$ with Gagliardo-Nirenberg inequality, we arrive at the following corollary on periodicity and regularity of $\ull$.
\begin{corollary}\label{ul:cor}
Consider the velocity field $\ull$ in (\ref{sub:ul}) subject to sub-critical initial data $(h_0,\vu_0)\in H^m({\T^2})$ with $m>5$. Then, $\ull(t,\cdot)$ is a $2\pi\Rossby$-periodic in time, and the following upper bound, uniformly in time, holds,
\begin{subequations}\begin{equation*}
\|\ull-\ul\|_{m-2}\le\C {\Rossby\over\Froude}\left(1+{\Rossby\over\Froude}\right).
\end{equation*} 
In particular, since $\|\ul\|_m\le\C$ for subcritical $\Rossby$, 
we conclude that $\ull(t,\cdot)$ has the Sobolev regularity,
\begin{equation*}
\|\ull\|_{m-2}\le\C\left(1+{\Rossby\over\Froude}+{\Rossby^2\over\Froude^2}\right).
\end{equation*}\end{subequations}
\end{corollary}
We close this section by noting that  the second iteration led to an approximate RSW system linearized  around the pressuerless velocity field, $\ul$, (\ref{hl:eq}),(\ref{sub:ul}), which governs our improved, $2\pi\Rossby$-periodic  approximation, 
$(\hll(t,\cdot),\ull(t,\cdot))\in H^{m-1}(\T^2)\times H^{m-2}(\T^2)$.

\section{Long time existence of approximate periodic solutions}
\subsection{The shallow-water equations}\label{sec:4}
How close is $(\hll(t,\cdot),\ull(t,\cdot))$ to the exact solution $(h(t,\cdot),\vu(t,\cdot))$? Below we shall show that their distance, measured in $H^{m-3}({\T}^2)$, does not exceed ${\displaystyle \At}$. Thus, for sufficiently small $\dde$, the RSW solution $(h,\vu)$ is \emph{``approximately periodic''} which in turn implies its long time stability. This is the content of our main result.

\begin{theorem}\label{thm:main}
Consider the rotational shallow water (RSW) equations on a fixed 2D torus,
\begin{subequations} 
\begin{eqnarray} 
\pt h+\vu\cdot\nabla h+\left(\Froudei+h\right)\nabla\cdot \vu&=&0 \label{RSW2}\\ 
\pt \vu+\vu\cdot\nabla \vu+\Froudei\nabla h-\Rossbyi J\vu&=&0\label{RSW1} 
\end{eqnarray}  
\end{subequations} 
subject to sub-critical initial data $(h_0,\vu_0)\in H^m({\T^2})$ with $m>5$ and $\alpha_0:=\min(1+\Froude h_0(\cdot))>0$. 
Let 
\[
\dde={{\Rossby}\over\Froude^2}
\]
denote the ratio between the  Rossby number $\Rossby$ and the squared  Froude number $\Froude$, with subcritical $\Rossby \leq \Rossby_c(\nabla\vu_0)$ so that (\ref{CT}) holds. Assume $\Froude\le1$ for substantial amount of pressure forcing in (\ref{RSW1}). 
Then, there exists a constant $\C$, depending only on $m$, $\Rossby_c$, $\alpha_0$ and in particular depending linearly on $\|(h_0,\vu_0)\|_m$, such that the RSW equations admit a smooth,  ``approximate periodic'' solution in the sense that there exists a near-by $2\pi\Rossby$-periodic solution, $(\hll(t,\cdot),\ull(t,\cdot))$, such that
\begin{equation}\label{RSWspan}
\|p(t,\cdot)-\pl(t,\cdot)\|_{m-3}+ \|\vu(t,\cdot)-\ull(t,\cdot)\|_{m-3}\le\At.
\end{equation} 
Here $p$ is the ``normalized height'' such that $1+{1\over2}\Froude p=\sqrt{1+\Froude h}$, and correspondingly, $\pl$ satisfies $1+{1\over2}\Froude \pl=\sqrt{1+\Froude \hll}$. 

It follows that  the life span of the RSW solution, $t\lesssim t_{\dde}:=\ln(\dde^{-1})$ is prolonged due to the rapid rotation $\dde \ll 1$, and in particular, it tends to infinity when $\dde\rightarrow 0$. 
\end{theorem}

\begin{proof}
We compare the solution of the RSW system (\ref{RSW2}),(\ref{RSW1}) with the  solution, $(\hll,\ull)$, of  approximate RSW system (\ref{hl:eq}),(\ref{sub:ul}). To this end, we rewrite the latter in the equivalent form,
\begin{subequations}
\begin{eqnarray}
\label{R2}
\pt \hll+\ull\cn\hll+\left(\Froudei+\hll\right)\nabla\cdot\ull&=&(\ull-\ul)\cn\hll+\left(\Froudei+\hll\right)\nabla\cdot(\ull-\ul)\\
\label{R1}
\qquad \pt \ull+\ull\cn\ull+\Froudei\nabla\hll-\Rossbyi J\ull&=& 
(\ull-\ul)\cn\ull+R.
\end{eqnarray}
\end{subequations}

The approximate system differs from the exact one, (\ref{RSW2}),(\ref{RSW1}),  in the residuals on the RHS of (\ref{R2}),(\ref{R1}). We will show that they  have an amplitude of order ${\dde}$. In particular, the comparison in the  rotationally dominant regime, $\dde \ll 1$ leads to a long-time existence of a smooth RSW solution which remains ``nearby" the time-periodic solution,
 $(\hll,\ull)$. To show that $(\hll,\ull)$ is indeed an approximate solution for the RSW equations, we proceed as follows.

We first  symmetrize the both systems so  that we can employ the standard energy method for nonlinear hyperbolic systems. To this end, We set the new variable (``normalized height'') $p$ such that $1+{1\over2}\Froude p=\sqrt{1+\Froude h}$.  Compressing notations with $\U:=(p,\vu)^\top$, we transform (\ref{RSW2}),(\ref{RSW1}) into the \emph{symmetric hyperbolic} quasilinear system
\begin{equation}\label{U}
\pt \U+B(\U,\nabla \U)+K[\U]=0.
\end{equation} 
Here $B(\F,\nabla \G):=A_1(\F)\G_x+A_2(\F)\G_y$ where $A_1,A_2$ are bounded linear functions with values being symmetric matrices, and $K[\F]$ is a skew-symmetric linear operator so that $\langle K[\F],\F\rangle=0$. By standard energy arguments, e.g. \cite{Ka75},\cite{KM82},\cite{KL04}), the symmetric form of (\ref{U}) yields an exact RSW solution $\U$, which stays smooth for finite time $t \lesssim 1$. The essence of our main theorem is that for small $\dde$'s, rotation  \emph{prolongs} the life span of classical solutions up to $t\sim {\mathcal O}(\ln\dde^{-1})$. To this end, we symmetrize the  approximate system (\ref{R2}), (\ref{R1}), using a new variable $\pl$ such that $1+{1\over2}\Froude\pl=\sqrt{1+\Froude\hll}$. Compressing notation with $\Vl:=(\pl,\ull)^\top$, we have 
\begin{equation}\label{V1}
\pt \Vl+B(\Vl,\nabla\Vl)+K(\Vl)=\bR
\end{equation} 
where the residual $\bR$ is given by 
\[
\bR:=\left[\ba{rcl}(\ull-\ul)\cn\pl&+&\left(\frac{2}{\Froude}+\pl\right)\nabla\cdot(\ull-\ul)\\
(\ull-\ul)\cn\ull&-&R\ea\right],
\] 
with $R$ defined in (\ref{Ru}).
\noindent
We will show $\bR$ is small which in turn, using the symmetry of (\ref{U}) and (\ref{V1}), will imply that $\|\U-\Vl\|_{m-3}$ is equally small. Indeed, thanks to the fact that $H^{m-3}(\T^2)$ is an algebra for $m>5$, every term in the above expression is upper-bounded in $H^{m-3}$, by the quadratic products of the terms, $\|\ul\|_m,\|\pl\|_{m-1},\|\ull\|_{m-2},\|\ull-\ul\|_{m-2}$, up to a factor of ${\mathcal O}(1+\Froudei)$.
The Sobolev regularity of these terms, $\ul,\ull$ and $\pl$ is guaranteed, respectively, in corollary \ref{CTcor}, corollary \ref{ul:cor} and theorem \ref{hl:th}. Moreover, the non-vacuum condition, $1+\Froude h_0\ge\alpha_0>0$, implies that $1+\Froude \hll$ remains uniformly bounded from below, and by standard arguments (carried out in Appendix A), 
$\|\pl\|_{m-2}\le\C( 1+\Rossby/\Froude)$.    
Summing up, the residual $\bR$ does not exceed,
\begin{equation}\label{Rineq}
\|\bR\|_{m-3}\le\C^2\left(\dde+ \frac{\Rossby}{\Froude}+...+{\Rossby^4\over\Froude^4}\right)\lesssim \C^2\dde,
\end{equation} 
for sub-critical $\Rossby\in(0,\Rossby_c)$ and under scaling assumptions $\dde<1$, $\Froude<1$.

We now claim that the same ${\mathcal O}(\dde)$-upperbound holds for the error $\W:=\Vl-\U$, 
for a long time, $t\lesssim t_\dde$.
Indeed, subtracting (\ref{U}) from (\ref{V1}), we find the error equation
\[
\pt \W+B(\W,\nabla \W)+K[\W]=-B(\Vl,\nabla \W)-B(\W,\nabla\Vl)+\bR.
\]
By the standard energy method using integration by parts and Sobolev inequalities while utilizing the symmetric structure of $B$ and the skew-symmetry of $K$, we arrive at
\[
{d\over dt}\|\W\|^2_{m-3}\lsm\|\W\|^3_{m-3}+\|\Vl\|_{m-2}\|\W\|^2_{m-3}+\|\bR\|_{m-3}\|\W\|_{m-3}.
\]

Using the regularity estimates of $\Vl=(\pl, \ull)^\top$ and the upper bounds on $\bR$ in (\ref{Rineq}), we end up with an energy inequality for $\|\W(t,\cdot)\|_{m-3}$,   
\[
{d\over dt}\|\W\|_{m-3}\lsm \|\W\|_{m-3}^2+\C\|\W\|_{m-3}+\C^2\dde, \quad\|\W(0,\cdot)\|_{m-3}=0.
\]
A straightforward integration of this forced Riccati equation (consult for example, \cite[\S5]{LT02}), shows that the error $\|\W\|_{m-3}$ does not exceed
\begin{equation}\label{T} 
\|\U(t,\cdot)-\Vl(t,\cdot)\|_{m-3}\le\At.
\end{equation}
In particular, the RSW equations admits an ``approximate periodic" $H^{m-3}({\T}^2)$-smooth solutions for $t\le{1\over \C} \ln (dde^{-1})$ for $\dde \ll 1$. 
\end{proof}
\begin{remark}\label{hhl:re}The estimate on the actual height function $h$ follows by applying the Gagliardo-Nirenberg inequality to $h-\hll=p(1+{\Froude\over4} p)-\pl(1+{\Froude\over4}\pl)=(p-\pl)(1+{\Froude\over4}(p-\pl)+{\Froude\over2}\pl)$,
\[
\|h(t,\cdot)-\hll(t,\cdot)\|_{m-3}\lesssim{e^{\C t}\dde\over(1-e^{\C t}\dde)^2}.
\] 

\end{remark}
Our result is closely related to observations of the so called ``near-inertial oscillation'' (NIO) in oceanography (e.g. \cite{YB97}). These NIOs are mostly seen after a storm blows over the oceans. They exhibit almost periodic dynamics with a period consistent with the Coriolis force and stay stable for about 20 days which is a long time scale relative to many oceanic processes such as the storm itself. This observation agrees with our theoretical result regarding the stability and periodicity of RSW solutions. In terms of physical scales, our rotationally dominant condition, ${\displaystyle \dde={gH\over fLU}\ll1},$ provides a physical characterization of this phenomenon. Indeed, NIOs are triggered when storms pass by (large $U$) and only a thin layer of the oceans is reactive (small aspect ratio $H/L$). Upon using the multi-layer model (\cite[\S 6.16]{Pe92}), we consider scales $f=10^{-4}s^{-1}, L=10^5m, H=10^2m, U=1ms^{-1}, g=0.01ms^{-2}$ (reduced gravity due to density stratification --   consult \cite[\S 1.3]{Pe92}). With this parameter setting, 
$\dde=0.1$, and theorem \ref{thm:main}  implies the existence of smooth, approximate periodic solution over time scale  $\ln(\dde^{-1})L/U\approx2$ days. We note in passing that  most cyclonic storms on the Northern Hemisphere rotates clockwise,  yielding a negative vorticity, $\omega_0=\partial_yu_0-\partial_xv_0<0$, which is a preferred scenario of the sub-critical condition (\ref{CT}) assumed in theorem \ref{thm:main}.

   
\subsection{The isentropic gasdynamics}\label{sec:5}
In this section we  extend theorem \ref{thm:main} to rotational 2D Euler equations for isentropic gas, 
\begin{subequations}\label{sub:EUph}
\begin{eqnarray}\label{EUph}
\pt \rho+\nabla\cdot(\rho \vu)&=&0\\
\pt \vu+\vu\cn \vu+\rho^{-1}\nabla \ppi(\rho)-f\vu^{\perp}&=&0.\label{EUpu}
\end{eqnarray}
Here, $\vu:=(u^{(1)}, u^{(2)})^\top$ is the velocity field, $\rho$ is the density and 
$\ppi=\ppi(\rho)$ is the pressure which  for simplicity, is taken to be that of a 
polytropic gas, given by the $\gamma$-power law, 
\begin{equation}\label{isenp}
\ppi(\rho)=A\rho^\gamma.
\end{equation}
\end{subequations}
The  particular case $A=g/2, \gamma=2$, corresponds to the RSW equations (\ref{phys2}),(\ref{phys1}). The following argument for long term existence of the 2D rapidly rotating isentropic equations applies,  with minor modifications,  to the  more general  pressure laws, $\ppi(\rho)$, which induce the hyperbolicity of (\ref{EUph}).

We first transform the isentropic Euler equations (\ref{EUph}) into their nondimensional form,
\begin{subequations}\label{isen:ab}
\begin{eqnarray*}
\pt \rho+\vu\cn\rho+\left(\Froudei+\rho\right)\nabla\cdot \vu&=&0 \label{isen:a}\\
 \pt \vu+\vu\cn \vu+\frac{1}{\Froude^{2}}\nabla(1+\Froude\rho)^{\gamma-1}-\Rossbyi J\vu&=&0 \label{isen:b}
\end{eqnarray*}
\end{subequations}
where the Mach number $\Froude$ plays the same role as the Froude number in the RSW equation. In order to utilize the technique developed in the previous section, we introduce a new variable $h$ by setting $1+\Froude h=(1+\Froude\rho)^{\gamma-1}$, so that the new variables, $(h,\vu)$, satisfy 
\begin{subequations}\label{sub:EU12}
\begin{eqnarray}\label{EU1}
\pt h+\vu\cn h+(\gamma-1)\left(\Froudei+h\right)\nabla\cdot \vu&=&0,\\
\label{EU2}\pt \vu+\vu\cn \vu+\Froudei\nabla h-\Rossbyi J\vu&=&0.
\end{eqnarray}
\end{subequations}
This is an analog to the  RSW equations (\ref{RSW2}),(\ref{RSW1}) except for the  additional factor $(\gamma-1)$ in the mass equation (\ref{EU1}). We can therefore  duplicate the steps which led to theorem \ref{thm:main} to obtain a long time existence  for the rotational Euler equations (\ref{EU1}),(\ref{EU2}). We proceed as follows.

An approximate solution is constructed in two steps. First, we use the $2\pi\Rossby$-periodic  pressureless solution, $(\ho\equiv\mbox{constant}, \ul(t,\cdot))$ for sub-critical initial data, 
$(h_0,\vu_0)$. Second, we construct a $2\pi\Rossby$-periodic solution $(\hll(t,\cdot),\ull(t\cdot))$ as the solution to an approximate  system of the  isentropic equations, \emph{linearized} around the pressureless velocity $\ul$,
\begin{eqnarray*}
\pt \hll+ \ul\cdot\nabla\hll + (\gamma-1)\left(\Froudei+\hll\right)\nabla\cdot\ul=0,\\
\ull:= \ul + \frac{\Rossby}{\Froude}J\left(I-\etj\right)\nabla\hll(t,\cdot).
\end{eqnarray*}
In the final step, we compare $(h,\vu)$ with the $2\pi\Rossby$-periodic approximate solution,
$(\hll,\ull)$. To this end, we symmetrize the corresponding systems using $\U=(p,\vu)^\top$ with
the normalized density function $p$ satisfying $1+{1\over2}\sqrt{1\over \gamma-1}\Froude p=\sqrt{1+\Froude h}$.
Similarly, the approximate system is symmetrized with the variables $\Vl=(\pl,\ull)$ where
$1+{1\over2}\sqrt{1\over \gamma-1}\Froude \pl=\sqrt{1+\Froude \hll}$.
We conclude

\begin{theorem}\label{thm:main1}
Consider the rotational isentropic equations on a fixed 2D torus, (\ref{isen:a}), (\ref{isen:b}),
subject to sub-critical initial data $(\rho_0,\vu_0)\in H^m({\T^2})$ with $m>5$ and $\alpha_0:=\min(1+\Froude \rho_0(\cdot))>0$. 

\noindent
Let 
\[
\dde={{\Rossby}\over\Froude^2}
\]
denote the ratio between the Rossby and the squared Mach numbers, with  sub-critical 
 $\Rossby \leq \Rossby_c(\nabla\vu_0)$ so that (\ref{CT}) holds. Assume $\Froude<1$ for substantial amount of pressure in (\ref{isen:b}). 
Then, there exists a constant $\C$, depending only on $m$, $\|(\rho_0,\vu_0)\|_m$, $\Rossby_c$ and $\alpha_0$, such that the RSW equations admit a smooth,  ``approximate periodic'' solution in the sense that there exists a near-by $2\pi\Rossby$-periodic solution, $(\rholl(t,\cdot),\ull(t,\cdot))$ such that
\begin{equation}\label{span:isen}
\|p(t,\cdot)-\pl(t,\cdot)\|_{m-3}+ \|\vu(t,\cdot)-\ull(t,\cdot)\|_{m-3} \le\At.
\end{equation}
Here, $p$ is the normalized density function satisfying $1+\Froude p=(1+\Froude\rho)^{\gamma-1\over2}$, and $\pl$ results from the same normalization for $\rholl$. 

It follows that  the life span of the isentropic solution, $t\lesssim t_{\dde}:=1+\ln(\dde^{-1})$ is prolonged due to the rapid rotation $\dde \ll 1$, and in particular, it tends to infinity when $\dde\rightarrow 0$. 
\end{theorem}
\begin{remark}\label{rho:re}For the actual density functions, $\rho-\rholl=\Froudei[(1+\Froude p)^{2\over\gamma-1}-(1+\Froude \pl)^{2\over\gamma-1}]=\int_0^1\Cg[1+\Froude(\theta(p-\pl)+\pl)]^{{2\over\gamma}-1}\,d\theta$
\[
\|\rho(t,\cdot)-\rholl(t,\cdot)\|_{m-3}\lesssim{e^{\C t}\over(1-e^{\C t})^{2\over\gamma-1}},
\]
in the physically relevant regime $\gamma\in(1,3)$.
\end{remark}

\subsection{The ideal gasdynamics}\label{sec:6}
We turn our attention to the full Euler equations in the 2D torus,

\begin{subequations}\label{Euler}
\begin{eqnarray*}
\pt \rho+\nc(\rho \vu)&=&0, \label{Euler1}\\
\pt \vu+\vu\cn \vu+\rho^{-1}\nabla \ppi(\rho,S)&=&f J\vu,
 \label{Euler2}\\
\pt S+\vu\cn S&=&0,\label{Euler3}
\end{eqnarray*}
\end{subequations}
where the pressure law is given as a function of the density, $\rho$ and the specific  entropy $S$, $\ppi(\rho,S):=\rho^\gamma e^{S}$. It can be symmetrized by defining a new variable -- the ``normalized'' pressure function,
\[
p:={\sqrt{\gamma}\over\gamma-1}\ppi^{\gamma-1\over2\gamma},
\]
and by replacing the density equation (\ref{Euler1}) by a (normalized) pressure equation, so that the above system is recast into an equivalent and symmetric form,
e.g., \cite{Ka75},\cite{Gr98}
\begin{eqnarray*}e^S\pt p+e^S\vu\cn p+\Cg e^Sp\nc \vu&=&0,\\
\pt \vu+\vu\cn \vu+\Cg e^Sp\nabla p&=&fJu, \qquad \Cg:={\gamma-1\over2},\\
\pt S+\vu\cn S&=&0.\end{eqnarray*}
It is the exponential function, $e^S$, involved in \emph{triple} products such as $e^Sp\nabla p$, that makes the ideal gas system a nontrivial generalization of the RSW and isentropic gas equations.

We then proceed to the nondimensional form by substitution,
\[
\vu\ra\mbox{U} u',\quad p\ra\mbox{P}(1+\Froude p'),\quad S=\ln(p\rho^{-\gamma})\ra\ln(\mbox{P}\mbox{R}^{-\gamma})+\Froude S'
\] 
After discarding all the primes, we arrive at a nondimensional system
\begin{subequations}\label{nEuler}
\begin{eqnarray}
e^{\Froude S}\pt p+e^{\Froude S}\vu\cn p+\Cg\left({e^{\Froude S}-1\over\Froude}+e^{\Froude S}p\right)\nc \vu&=&-\Cg\Froudei\nc \vu,\\
\label{nEuler:u}\pt \vu+\vu\cn \vu+\Cg\left({e^{\Froude S}-1\over\Froude}+e^{\Froude S}p\right)\nabla p&=&-\Cg\Froudei\nabla p+\Rossbyi Ju,\\
\pt S+\vu\cn S&=&0,
\end{eqnarray}
\end{subequations}
where $\Froude$ and $\Rossby$ are respectively, the Mach and the Rossby numbers. With abbreviated notation, $\U:=(p,\vu,S)^\top$, the equations above amount to a symmetric hyperbolic system written in the compact form,
\begin{equation}\label{nEuler:comp} 
A_0(S)\pt\U+A_1(\U)\partial_x\U+A_2(\U)\partial_y\U=K[\U].
\end{equation}
Here, $A_i (i=0,1,2)$ are symmetric-matrix-valued functions, {\it nonlinear} in $\U$ and in particular $A_0$ is always positive definite. The linear operator $K$ is skew-symmetric so that $\langle K[\U],\U\rangle=0$.

Two successive approximations are then constructed based on the iterations (\ref{suc}), starting with $j=1$,
\begin{eqnarray*}
p_1&\equiv&\mbox{constant},\\
\pt\ul+\ul\cn\ul&=&\Rossbyi J\ul,\\
S_1&\equiv&\mbox{constant}.
\end{eqnarray*}
Identified as the pressureless solution, $\ul$ is used to linearize the system, resulting in the following approximation
\begin{subequations}\label{nEuler2}
\begin{eqnarray}
\pt \pl+\ul\cn\pl+\Cg\pl\nc\ull&=&-\Cg\Froudei\nc\ull,\\
\ull-\ul&=&{\Rossby\over\Froude}J(I-\etj)\Cg e^{\Froude S_2}(1+\Froude p_2)\nabla p_2,\\
\pt S_2+\ul\cn S_2&=&0
\end{eqnarray}
\end{subequations}

The $2\pi\Rossby$-periodicity and global regularity of $\U_2:=(p_2,\vu_2,S_2)^\top$ follow along  the same lines outlined for the RSW equations in section \ref{sec:3} (and therefore omitted), together with the following nonlinear estimate for $e^{\Froude S}$,
\[
\|e^{\Froude S}-1\|_m=\left\|\sum^\infty_{j=1}{({\Froude S})^j\over j!}\right\|_m
\lsm\sum^\infty_{j=1}{(C_m|{\Froude S}|_\infty)^{j-1}\over j!}\|{\Froude S}\|_m
={e^{C_m|{\Froude S}|_\infty}-1\over C_m|{\Froude S}|_\infty}\|{\Froude S}\|_m;
\]
for the latter, we  apply recursively the Gagliardo-Nirenberg inequality to typical terms  $\|({\Froude S})^j\|_m$. Notice the entropy variable (both the exact and approximate ones) always satisfies a transport equation and therefore is conserved along particle trajectories, which implies that the $L^\infty$ norm of the entropy variable is an invariant. Thus, we arrive at an estimate
\begin{equation}\label{exp:est}
\|e^{\Froude S}-1\|_m\le\Froude\Cinf\|S\|_m.
\end{equation}  
Of course, the same type of estimate holds  for the approximate entropy, $S_2$.

Finally, we subtract the approximate system (\ref{nEuler2}) from the exact system (\ref{nEuler:comp}), arriving at an error equation for $\W:=\U-\U_2$ that shares the form as for the RSW system in Section \ref{sec:4}, except that $A_i(\U)-A_i(\U_2)\neq A_i(\U-\U_2)$ due to nonlinearity which is essentially quadratic in the sense that, \footnote{Consider a typical term of $A_i$, e.g. $e^{\Froude S}p$. Applying (\ref{exp:est}) together with Gagliardo-Nirenberg inequality to $e^{\Froude S}-e^{\Froude S_2}=e^{\Froude S_2}(e^{\Froude(S-S_2)-1})$, we can show $\|e^{\Froude S}-e^{\Froude S_2}\|_n\lesssim\|S-S_2\|_n$. The estimate on $\|e^{\Froude S}p-e^{\Froude S_2}\pl\|_n$ then follows by applying identity $ab-a_2b_2=(a-a_2)(b-b_2)+(a-a_2)b_2+a_2(b-b_2)$ together with the triangle inequality and the G-N inequality. Here regularity of $S_2$ and $p_2$ is a priori known.} 
\[
\|A_i(\U)-A_i(\U_2)\|_n\lesssim\|\U-\U_2\|^2_n+\|\U-\U_2\|_n,\qquad i=0,1,2,
\]
\[
\|A_i(\U)-A_i(\U_2)\|_{W^{1,\infty}}\lesssim\|\U-\U_2\|^2_{W^{1,\infty}}+\|\U-\U_2\|_{W^{1,\infty}},\qquad i=0,1,2.
\]
where $n>2$. This additional nonlinearity manifests itself as three more multiplications in the energy inequality,
\[
{d\over dt}\|\W\|_{m-3}\lesssim\|\W\|^5_{m-3}+...+\|\W\|_{m-3}+\dde,\qquad\|\W(0,\cdot)\|_{m-3}=0,
\]
whose solution (-- developed around a simple root of the quintic polynomial on the right), has the same asymptotic behavior as for the quadratic Riccati equations derived in the previous sections.
  
\begin{theorem}\label{thm:main2}
Consider the (symmetrized) rotational Euler equations on a fixed 2D torus (\ref{nEuler}) 
subject to sub-critical initial data $(p_0,\vu_0,S_0)\in H^m({\T^2})$ with $m>5$.

\noindent
Let 
\[
\dde={{\Rossby}\over\Froude^2}
\]
denote the ratio between the Rossby and the squared Mach numbers, with 
subcritical  $\Rossby \leq \Rossby_c(nabla\vu_0)$ so that (\ref{CT}) holds. Assume $\Froude<1$ for substantial amount of pressure forcing in (\ref{nEuler:u}). 
Then, there exists a constant $C$, depending only on $m$, $\|(p_0,\vu_0,S_0)\|_m$, $\Rossby_c$, such that the ideal gas equations admit a smooth,  ``approximate periodic'' solution in the sense that there exists a near-by $2\pi\Rossby$-periodic solution, $(\pl(t,\cdot),\ull(t,\cdot),S_2(t,\cdot))$ such that
\begin{equation}\label{span:gas}
\|p(t,\cdot)-\pl(t,\cdot)\|_{m-3}+ \|\vu(t,\cdot)-\ull(t,\cdot)\|_{m-3}+\|S(t,\cdot)-S_2(t,\cdot)\|_{m-3} \le\At.
\end{equation} 
It follows that  the life span of the ideal gas solution, $t\lesssim t_{\dde}:=\ln(\dde^{-1})$ is prolonged due to the rapid rotation $\dde \ll 1$, and in particular, it tends to infinity when $\dde\rightarrow 0$. 
\end{theorem}

\section{Appendix. Staying away from vacuum}\label{app:A}
We will show the following proposition on the new variable $\pl$ defined in section \ref{sec:4}.

\begin{proposition}
Let $\pl$ satisfies 
\begin{equation}\label{qi}1+{1\over2}\Froude \pl=\sqrt{1+\Froude \hll}\end{equation} where $\hll$ is defined as in (\ref{hl:eq}), that is,
\begin{equation}\label{h11}
\pt \hll +\ul\cn\hll+\left(\Froudei+\hll\right)\nc\ul=0
\end{equation}
subject to initial data $\hll(0,\cdot)=h_0(\cdot)$ that satisfies the non-vacuum condition $1+\Froude h_0(\cdot)\ge\alpha_0>0$. Then,
\[
|\pl|_\infty\le\Cinf\left( 1+{\Rossby\over\Froude}\right),
\]
\[
\|\pl\|_n\le\C\left( 1+{\Rossby\over\Froude}\right).
\]
\end{proposition}
The proof of this proposition follows two steps. First, we show that the $L^\infty$ and $H^n$ norms of $\pl(0,\cdot)$ are dominated by $\hll(0,\cdot)$ due to the non-vacuum condition. Second, we derive the equation for $\pl$ and obtain regularity estimates using similar techniques from section \ref{sec:4}.

Step 1. For simplicity, we use $p:=\pl(0,\cdot)$ and $h:=\hll(0,\cdot)$.

Solving (\ref{qi}) and differentiation yield
\[
p={2 h\over\sqrt{1+\Froude h}+1}, \quad \nabla p={\nabla h\over\sqrt{1+\Froude h}}.
\]
Clearly,  $|p|_\infty\le|h|_\infty$.
The above identities, together with the non-vacuum condition imply 
\[
\| p\|_1\le  2\|h\|_1 \quad\mbox{ and }\quad |\nabla p|_{L^\infty}\le {|\nabla h|_{L^\infty}\over\sqrt{\alpha_0}}.
\]

For higher derivatives of $p$, we use the following recursive relation. Rewrite (\ref{qi}) as $ p+{1\over4}\Froude p^2= h$ and then take the $k$-th derivative on both sides 
\[
D^k p+{1\over4}\Froude2 p D^k p+{1\over4}\Froude\left(D^k(q^2)-2 p D^k p\right)=D^k h
\]
so that taking $L^2$ norm of this equation yields
\[
I-II:=\left\|(1+{1\over2}\Froude p)D^k p\right\|_0-{1\over4}\Froude\left\|D^k(q^2)-2 p D^k p\right\|_0\le\|D^k h\|_0.
\]
Furthermore, we find $I\ge\sqrt{\alpha_0}\|D^k p\|_0$ by (\ref{qi}) and the non-vacuum condition. We also find $II\lesssim_n|\nabla p|_\infty\| p\|_{|k|-1}$ by  Gagliardo-Nirenberg inequalities. Thus we arrive at a recursive relation 
\[
\|p\|_{|k|}\le\Cinf (\|p\|_{|k|-1}+\|h\|_{|k|})
\]
 which implies that the $H^n$ norm of $\pl(0,\cdot)=p$ is dominated by $\|\hll(0,\cdot)\|_n=\|h\|_n$.     

Step 2. We derive an equation for $\pl$ using relation (\ref{qi}) and equation (\ref{h11}),
\[
\pt \pl+2\ul\cn\pl+\left(\Froudei+\pl\right)\nc\ul=0.
\]
This equation resembles the formality of the approximate mass equation (\ref{hl:eq}) for $\hll$ and thus we apply similar technique to arrive at the same regularity estimate for $\pl$, 
\begin{eqnarray*}
|\pl(t,\cdot)|_\infty&\le&\Cinf\left( 1+{\Rossby\over\Froude}\right),\\
\|\pl(t,\cdot)\|_n&\le&\C\left( 1+{\Rossby\over\Froude}\right).
\end{eqnarray*}

\end{document}